\documentclass{article}
\usepackage{geometry}
\usepackage[colorlinks = true,linkcolor = blue,urlcolor  = blue,citecolor = blue,anchorcolor = blue]{hyperref}

\usepackage{enumerate, graphicx, float, caption, ulem, mathrsfs, amsmath, amsthm, amssymb, amsfonts}

\def\E{\mathbb E}
\def\dif{\mathrm d}
\def\P{\mathbb P}

\def\ep{\varepsilon}
\def\F{\mathcal F}
\def\L{\mathcal L}

\newtheorem{thm}{\noindent Theorem}[section]
\newtheorem{lem}{\noindent Lemma}[section]
\newtheorem{prop}{\noindent Proposition}[section]
\newtheorem{assumpt}{\noindent Assumption}[section]
\newtheorem*{acknowledgement}{Acknowledgement}
\newtheorem{rem}{\noindent Remark}[section]

\title{Learning from the past in an irreversible investment problem}
\author{Topias Tolonen-Weckström\\\textit{Department of Mathematics, Uppsala University}\thanks{Email address: \href{mailto: topias.tolonen@math.uu.se}{topias.tolonen@math.uu.se.} This manuscript is accepted for publication and is to appear in: \textit{Journal of Applied Probability} 64 (3), 2027.}}
\date{November 5, 2026}

\begin{document}

\maketitle

\begin{abstract}
We consider an irreversible investment problem under incomplete information, where the investor decides whether and when to make investments in a project. Upon investment, the investor acquires previously hidden information from the project's past (''learning from the past''), and so the learning rate of the problem is controlled by investing. We set up this original problem as a recursively defined stopping problem  with a learning-from-the-past effect. We show that, under general conditions, at each recursion step there exists a one-sided stopping boundary. We proceed to present the optimal investment strategy as a sequence of semi-explicit stopping boundaries derived from the smooth fit conditions  of the corresponding free-boundary problems. Feasibility of our approach is then demonstrated by solving example boundaries numerically and by performing sensitivity analysis with respect to different model parameters.
\bigskip

\bigskip


\end{abstract}

\section{Introduction}\label{ch1}
Consider a Bayesian decision-maker (investor) whose objective is to decide whether and when to make irreversible investments in a project. The decision-maker makes noisy observations of the project value under incomplete information and knows that after each investment, they will instantaneously learn more about the project. We call this notion \textbf{learning from the past} and it is the main point of interest in our study. 

We describe such a problem informally as
\begin{equation}\label{introdiffusioneq}
\begin{cases}
    \dif X_t = \mu\dif t+ \sigma\dif W_t\\
     Y_t = \int_{-\delta U_t}^t\dif X_t,
\end{cases}   
\end{equation}
where the observation process $X_t$ represents a noisy observation stream of the project's value with $\sigma>0$, $W_t$ is a two-sided Brownian motion, and an unknown project value $\mu$ assumes a Bernoulli distribution with two possible values $\mu_0<0<\mu_1$. In~\eqref{introdiffusioneq}, $Y_t$ includes additional observations of $X_t$ (the \textit{learning-from-the-past} effect), $\delta>0$ is a constant denoting amount of learning per unit of investment, and $(U_t)_{t\geq0}$ is an increasing control process that denotes a share of investment in a project with $U_0=0$ and $U_t\leq 1$. The objective of the investor is then to control $U_t$ to maximize 
\begin{equation}\label{introvalueftion}
\E\left[\int_0^\infty e^{-rt}\mu \dif U_t\right],
\end{equation}
where $r$ is a known discount rate. Additional observations included in $Y_t$ represent the investment's past and they are realized to the decision-maker instantly (i.e. outside the physical calendar time $t$ that supports observations of $X_t$), and thus the additional observation period is not discounted.

The learning-from-the-past effect arises naturally in applications of irreversible investment problems. In particular, it is a natural model for cases of \textbf{project acquisition}, where the investor, upon acquisition, learns about the project's (for example, a company) past operations that culminate as hidden intangible assets within the project that are not accounted for in its public financial statements. When accessing the project as an insider, the investor gains insider information about such assets, which can include previously unaccounted and publicly hidden goodwill, intellectual property, human economic value (human resources and/or sweat equity), and organizational culture.

From an economic standpoint, the project value $\mu$ represents the present-value utility a fully-invested decision-maker cumulatively acquires over an infinite time horizon (similar interpretation is seen, for example, in~\cite{DP}). However, $\mu$ is not observable and represents an abstract intrinsic value of the project, similarly as in~\cite{LBD}. To focus on modeling the learning effect, we consider the simplest Bayesian setting with only two possible values of $\mu$. The model choice of $\mu_0<0<\mu_1$ results in an inherent risk for the investor: investing in a project with $\mu=\mu_0<0$ is costly to the investor. The model includes no other cost of investment. Instead, the model can be extended to include investment costs $C>0$ by replacing $\mu$ in~\eqref{introvalueftion} with $\mu^*:=\mu-C$. Assuming $\mu_0<C<\mu_1$, $\mu^*$ then is an abstract representation of a return on investment.

The value~\eqref{introvalueftion} assumes a natural economic interpretation. Consider a project with value $\mu$ at time $t>0$. Increasing investments in the project by $\dif U_t$ yields a present-value utility of $e^{-rt}\mu\dif U_t$. Integrating over all future incremental investments leads to the formulation in~\eqref{introvalueftion}.

Intuitively, the problem of maximizing~\eqref{introvalueftion} is solved by finding a suitable stopping boundary. However, we note that there are complications in the above problem formulation as the information available to the investor depends on the control process $U_t$, which in turn should depend on the available information, creating a circular feedback between the observation process and admissible controls adapted to it. Also, the continuous formulation is challenging as it leads to rather involved regularity considerations along the curved boundary. Instead of formulating this problem precisely, in order to gain mathematical tractability and in line with most real-world applications, we choose to study a discrete version of the problem.

In particular, we restrict the possible investment levels to only attain discrete values so that $U_t$ changes only upon investment and that $U_t\in\{u_0,u_1,\ldots,u_N\}$ with $\{u_n\}_{n=0}^N$ decreasing, $u_N=0$, and $u_0=1$ that corresponds to being fully invested in a project. Here the index $n$ in $u_n$ indicates that there are $n$ remaining investment possibilities, so $u_N$ is the initial level of investment and $u_0=1$ is the investment level after the last possible investment. Under such restriction, the control problem described in~\eqref{introdiffusioneq}--\eqref{introvalueftion} collapses into a recursive stopping problem. 

In this recursive stopping problem with a learning-from-the-past effect, upon (partial) investment, we let additional observations corresponding to time $\delta/N$ become available. The problem is solved by using backward induction, so we let the optimal stopping problem corresponding to the $N$th investment (starting with investment level $u_1$ and increasing it to $u_0=1$) be the investor's starting problem. In line with the continuous problem~\eqref{introvalueftion}, the investor finds an optimal investment time $\tau$ to attain
\begin{equation}\label{introinitialvalue}
    \sup_\tau\E\left[e^{-r\tau}\mu\Delta u_1\right],
\end{equation}
where $\Delta u_1:=u_1-u_0$.

To solve the subsequent investment times of the remaining $N-1$ investments, upon each investment, the investor solves an optimal stopping problem of the type 
\begin{equation}\label{discrete-valueftion-mu}
\sup_{\tau}\E\left[e^{-r\tau}g(\cdot)\right],
\end{equation}
where $g(\cdot)$ is the sum of the value of that particular investment (of type $\mu\Delta u_n$ when there are $n$ investments remaining) and the expected value of all future investments. When evaluating the expected value of future investments, the investor needs to account for the $\delta / N$ time units of additional observations of $X_t$ that are realized upon investments. The recursive stopping problem~\eqref{introinitialvalue}--\eqref{discrete-valueftion-mu}, that is the main interest of this article and a discrete analogue of the control problem~\eqref{introdiffusioneq}--\eqref{introvalueftion}, 
is properly defined in Section~\ref{ch2}.

Our primary contribution is to set up and study an original problem of this type, deriving semi-explicit solutions for the optimal stopping problem. In our recursive problem formulation under incomplete information and additional learning, we show that, at each investment step where the future values of subsequent investments are enveloped, there exists a one-sided stopping boundary. To make the standard methods of optimal stopping go through, we provide a careful analysis of properties of different payoff functionals. We show that the boundaries characterizing the optimal investment times exist and are well-defined, as well as provide equations from which the boundaries can be solved numerically. Moreover, we demonstrate that the solution concept is feasible by providing numerical examples with sensitivity analysis for different model parameterizations.

\subsection{Related literature}

Our model with the learning-from-the-past effect is unique to this article. The model provides a new way of controlling learning rate in a optimal stopping problem in an irreversible investment setting under incomplete information. 

Investment and utility maximization problems under incomplete information are well studied in the field of stochastic control. An early contemporary reference that combines stochastic control with incomplete information is \cite{L}. A key study in early investment problems within the field is \cite{DMV2005}, where an investment timing problem under incomplete information with respect to an option payoff functional is studied. General investment timing problems are examined, for example, in \cite{DMV2006}, \cite{DP}, and \cite{MS}, while \cite{ferrari2015integral} examines an irreversible investment problem by characterizing the free boundary as the unique solution of an integral equation. Investment problems with a Bayesian setting under incomplete information are discussed for example in \cite{harrison2015investment} and \cite{sunar2021competitive}. In particular, \cite{sunar2021competitive} discusses the relationship between belief of a favorable market and investment timing, closely resembling our set-up. In a recent effort, \cite{gierens2025irreversibleinvestmentproblemincomplete} studies an irreversible investment problem under incomplete information, where the investment is modeled as a geometric Brownian motion. Our model resembles a set-up where an insider has access to additional information that is modeled via introducing larger filtrations than what would be accessible by only observing $X_t$. 
Such notion is also modeled in studies of stochastic control regarding insider trading, for example in \cite{draouiloks} and \cite{NunnoMBOksProske}.

Main point of interest in our model is that the decision-maker controls the learning rate of the observation process. Such stochastic control problems have only recently been discussed. Statistical problems of this type are considered in \cite{DS} with a problem of quickest detection with reversible controls, in \cite{EK} with an estimation problem with costly observations, and in \cite{EM2023} with a detection problem with irreversible controls and a linear cost to increase observation rate. \cite{LBD} incorporates an irreversible investment problem, where an investment directly affects the signal-to-noise ratio of the observation process.

We construct a recursively defined stopping problem that is studied \cite{CD} and \cite{carmonatouzi}. In \cite{CD} it is shown that under relatively general assumptions, the recursive stopping problem is solution concept for a multiple optimal stopping problem, and \cite{carmonatouzi} discusses multiple optimal stopping for American swing options, largely resembling our set-up.

The set-up in \cite{LBD} also models a learning feature in an irreversible investment problem. In their set-up, they consider an example of project expansion. Upon investing, the investor begins to learn at an accelerated rate due to gaining more capacities of observing the development of the market, product testing, and realized demand or production costs, for example by setting up a new production unit. There is a crucial difference between their learning and our learning-from-the-past effect. In our model, the key application considers learning by acquiring already established units. Moreover, opposed to their set-up, the investor can't directly affect the diffusion coefficient of the output process directly by increasing their investment level. Instead, the additional learning is modeled by speeding up the observation process upon investment. Realization of this accelerated process then reveals previously hidden information, inducing the learning-form-the-past effect. Despite these differences, both of the models study problems of irreversible investment under incomplete information, where the amount of learning is both controlled and monotone upon investment: investing more yields more information to the investor.

\subsection{Structure of the article}

The remainder of this article is organized as follows. In Section~\ref{ch2}, we set up the model and define the key concepts we use to build our model. More specifically, we introduce a recursive stopping problem with the learning-from-the-past effect. In Section~\ref{ch3}, a candidate solution is characterized in terms of an optimal investment strategy and the corresponding stopping boundaries. In Section~\ref{ch4}, we show that each investment step induces a one-sided stopping boundary. Main results are found in Section~\ref{ch5}, where verification result for each step is provided together with a main theorem which shows that individual verification results go through when combined recursively, resulting in an optimal investment strategy for our problem. Finally, we illustrate our theoretical results with key numerical examples in Section~\ref{ch6}.

\section{Problem set-up}\label{ch2}
Let $(X_t)_{t\in\mathbb R}$ be a diffusion process with dynamics
\begin{equation}\label{observationprocess}
    \dif X_t = \mu \dif t + \sigma \dif W_t,
\end{equation}
where $\mu$ is a random variable assuming two possible values with $\mathbb P(\mu=\mu_1)=\pi=1-\mathbb P(\mu=\mu_0)$, $\sigma>0$ and $X_0=x_0$ are known constants, and $W_t$ is a two-sided Brownian motion defined as
\begin{equation*}
    W_t = \begin{cases}
        W_1(t), \quad \text{ if } t\geq 0\\
        W_2(-t), \quad \hspace{-3mm}\text{ if } t< 0
    \end{cases},
\end{equation*}
where $W_1$ and $W_2$ are standard Brownian motions (independent of each other and $\mu$) defined on a sufficiently rich probability space $(\Omega,\F,\P)$.

We consider an investor who is facing an optimization problem of gradually investing to a project. The project value $\mu$ takes possible values $\mu_0$ and $\mu_1$ with $\mu_0<0<\mu_1$. We model incomplete information, i.e. the lack of the investor's information on $\mu$, by letting the investor observe realizations of $X_t$ instead of directly observing $\mu$.

Continuing with the notation presented in Section~\ref{ch1}, let the possible amounts of investment be $\{u_n\}_{n=0}^N$ and let $u_N$ denote the starting amount of investment and $u_0$ the final investment level. For this problem of irreversible investment, we assume the admissible investment amounts to be of the form
\begin{equation}\label{admissibleu}
    u_n = \frac{N-n}{N} \quad \text{ for } n\in\{0,1,\ldots,N\},
\end{equation}
and that upon each investment, the amount is raised by $1/N$. That is, the admissible investment amounts correspond to proportions of investment $\{u_n\}_{n=0}^N$ that is a decreasing sequence with $u_N=0$ and $u_0=1$.

We model a problem where increasing one's investments triggers additional observations of $X_t$. This is what we call the learning-from-the-past effect. To accommodate for the discrete set-up (cf. the continuous problem~\eqref{introdiffusioneq}), we restrict the amount of learning $\delta$ to 
\begin{equation}\label{deltarestriction}
    \delta=N\ep
\end{equation}
for some $\ep>0$, i.e. the amount of additional learning per partial investment is $\ep$. The simple model choice of a constant $\ep>0$ allows the investor to naturally favor earlier investments (as opposed to the amount of learning being dependent on the index $n$, which could be used to produce non-intuitive investment behavior and non-monotone stopping boundaries).

Let $\F^X_t$ be the completion of the $\sigma$-algebra $\sigma\{X_s:0\le s \le t\}$ and $\mathbb F^X=(\F^X_t)_{t\geq 0}$ its associated filtration. To characterize the investor's learning about $\mu$, we define the belief process of the investor as the conditional probability
\begin{equation}\label{piprocess}
    \Pi_t = \P(\mu=\mu_1 | \F_t^{X})
\end{equation}
with $\Pi_0 = \pi\in(0,1)$ being a known constant representing the investor's prior information that $\mu=\mu_1$, i.e. the probability $\P(\mu=\mu_1)$.

By a direct calculation (see, for example,~\cite{S1978} for corresponding calculations) one finds
\begin{equation}\label{piformula}
    \Pi_t = \frac{\pi e^{-\frac{\mu_1^2}{2\sigma^2}t+\frac{\mu_1}{\sigma^2}(X_t-X_0)}}{\pi e^{-\frac{\mu_1^2}{2\sigma^2}t+\frac{\mu_1}{\sigma^2}(X_t-X_0)}+(1-\pi)e^{-\frac{\mu_0^2}{2\sigma^2}t+\frac{\mu_0}{\sigma^2}(X_t-X_0)}} =:h(X_t-X_0,t).
\end{equation}
Moreover, from standard literature on filter theory, for example from \cite{liptser1977statistics}, we find that the dynamics of the process $\Pi_t$ can be described as
\begin{equation}\label{pidynamics}
    \dif \Pi_t = \rho\Pi_t(1-\Pi_t) \dif \tilde W_t,
\end{equation}
where 
\begin{equation}\label{innovationsprocess}
\tilde W_t := \frac{1}{\sigma}\left(X_t-\int_0^t (\mu_0+(\mu_1-\mu_0)\Pi_s)\,ds\right)
\end{equation}
is the so-called innovations process (an $\mathbb F^{X}-$Brownian motion) and $\rho:=\frac{\mu_1-\mu_0}{\sigma}$ is the signal-to-noise ratio.

It is well-known that $\Pi_t$ is a time-homogeneous process with a strong Markov property as it solves~\eqref{pidynamics} (see, for example,~\cite{oks}), and so we may embed the problem into a Markovian setting. Additionally, it follows that optimization over $\mathbb F^X-$stopping times coincides with optimization over $\mathbb F^{\Pi}-$stopping times ($\mathbb F^{\Pi}=(\F_t^\Pi)_{t\geq 0}$ being the filtration associated to $\F_t^\Pi$, the completion of $\sigma\{\Pi_s: 0\le s\le t\}$). 

Recall that the investor solves a recursive stopping problem with backward induction. Recalling the initial stopping problem~\eqref{introinitialvalue} presented in Section~\ref{ch1}, we let the investor's starting problem be
\begin{equation}\label{startingproblemmu}
    \sup_\tau\E\left[e^{-r\tau}\mu\Delta u_1\right],
\end{equation}
where $\tau$ is a $\mathbb F^\Pi-$stopping time. That is, the investor solves the problem of their $N$th investment to raise the investment level from $u_1=(N-1)/N$ to $u_0=1$, and no additional learning is imposed.

To solve this problem, note that for any time $t>0$, the project value $\mu$ satisfies
\begin{equation}\label{muest}
    \E\left[\mu\vert\F_t^X\right]=\Pi_t\mu_1+(1-\Pi_t)\mu_0.
\end{equation}
By using this estimator, the tower property of conditional expectation, and the strong Markov property of $\Pi_t$, and~\eqref{admissibleu}, it follows that
\begin{equation}\label{conditionaling}
\E\left[e^{-r\tau}\mu\Delta u_1\right] = \frac{(\mu_1-\mu_0)}{N}\E\left[e^{-r\tau}\left(\Pi_{\tau}-K\right)\right],
\end{equation}
where $K=-\mu_0/(\mu_1-\mu_0)$. Observe that optimizing over the left-hand and the right-hand sides of~\eqref{conditionaling} coincide. Consequently, the investor's starting problem~\eqref{startingproblemmu} becomes
\begin{equation}\label{precursorV1}
V_1(\pi) = \sup_{\tau}\E_\pi\left[e^{-r\tau}(\Pi_\tau-K)\right],
\end{equation}
now additionally noting for the initial value $\Pi_0=\pi$ in conditional expectation $\E_\pi\left[\cdot\right]:=\E\left[\cdot|\Pi_0 = \pi\right]$.

Investor's second stopping problem that corresponds to the $(N-1)$th investment is set up similarly. The investor starts at $u_2=(N-2)/N$ and solves a stopping problem to find $\tau$, a $\mathbb F^\Pi$-stopping time. Upon investment, they receive $e^{-r\tau}(\Pi_\tau-K)$ as well as the value of the remaining investment: the expected value of $V_1$ evaluated after noting for the learning-from-the-past effect.

Let $\Pi_\tau$ be the investor's belief process evaluated at time $\tau$. We model the learning-from-the past effect by letting the investor instantaneously acquire additional observations of $X_t$ from the time interval $[-\ep,0]$ after exercising a partial investment at time $\tau$ (cf.~\eqref{deltarestriction} and its surrounding paragraph). By the time-homogeneity of $\Pi_t$, this is the same as letting $\Pi_t$ diffuse for another $\ep$ time units. Indeed, upon investment, the investor's belief process is instantaneously updated from $\Pi_\tau=h(X_\tau-X_0,\tau)$ to $h(X_\tau-X_{-\ep},\tau+\ep)$, which coincides in law with $\Pi_{\tau+\ep}$. Therefore, the investor's second problem is
\begin{equation}\label{precursorV2}
    V_2(\pi) = \sup_\tau \E_\pi\left[e^{-r\tau}\left(\Pi_\tau-K+V_1(\Pi_{\tau+\ep})\right)\right],
\end{equation}
which is of delayed type (see, for example, \cite{oks2005delay} for general theory on such stopping problems). To work with the expected value of $V_1$, we recall~\eqref{pidynamics} and use tower property of conditional expectation together with the strong Markov property of $\Pi_t$ to obtain
\begin{align}\label{eq:V1expectation}
\E_\pi\left[e^{-r\tau}V_1(\Pi_{\tau+\ep})\right] &= \E_\pi\left[e^{-r\tau}V_1\left(\Pi_\tau+\int_{\tau}^{\tau+\ep} \rho\Pi_t(1-\Pi_t) \dif \tilde W_t\right)\right] \notag\\
&=\E_\pi\left[e^{-r\tau}\E_y\left[V_1\left(y+\int_{0}^\ep \rho\Pi_t(1-\Pi_t) \dif \tilde W_t\right)\bigg\rvert_{y=\Pi_\tau} \right]\right].
\end{align}

The first two recursive steps~\eqref{precursorV1} and~\eqref{precursorV2} of backward induction can be propagated for a total of $N$ steps up to the first investment (first with respect to the physical calendar time $t$), and so we are ready to present the investor's recursive stopping problem. 

Let $\mathfrak{F}$ be the collection of $\mathbb F^\Pi-$stopping times. Then, the investor solves

\begin{equation}\label{valueftion}
    \begin{cases}
            V_1(\pi) = \sup_{\tau\in\mathfrak{F}} \E_\pi\left[e^{-r\tau}g_1(\Pi_\tau)\right] \\
            V_n(\pi) = \sup_{\tau\in\mathfrak{F}} \E_\pi\left[e^{-r\tau}g_{n}(\Pi_\tau)\right]
    \end{cases}
\end{equation}
for $n\in\{2,\ldots,N\}$, where
\begin{equation}
    g_1(\pi):= \pi - K,
\end{equation}
\begin{equation}
    g_n(\pi):= \pi - K + F_{n-1}(\pi),
\end{equation}
and
\begin{equation}\label{F}
    F_{n-1}(\pi) := \E_\pi\left[V_{n-1}(Z(\ep,\pi)\right],
\end{equation}
where $Z(\ep,\pi):=\pi+\int_{0}^\ep \rho\Pi_t(1-\Pi_t) \dif \tilde W_t$ with $\Pi_0=\pi$.

In Sections~\ref{ch3}--\ref{ch5}, we treat the recursive problem~\eqref{valueftion}.

\begin{rem}\label{rmk:learning}
\textit{(Regarding the learning-from-the-past effect and the operator $F_n$).} In~\eqref{F}, the operator $F_{n-1}(\pi)$ (evaluated at $\Pi_\tau$) denotes a conditional expectation of $V_{n-1}$ evaluated over a value of a strong Markov process that has the dynamics of $\Pi_t$, starts from the value $\pi=\Pi_\tau$ (inherited from the conditional expectation), and diffuses for $\varepsilon$ time units. This operator allows us to treat the update of the belief process from $\Pi_\tau$ to $\Pi_{\tau+\ep}$ in each recursion step. In~\eqref{F}, $\Pi_{\tau+\ep}$ is denoted by $Z(\ep,\Pi_\tau)$ for us to treat it as a function of both $\ep$ and $\pi$.
\end{rem}

\section{Finding a candidate solution}\label{ch3}

We first have the following result.

\begin{lem}\label{lemmaFproperties}
    Let $g_n$, $V_n$ and $F_n$ be as in~\eqref{valueftion}--\eqref{F}. Then, for all $n=1,\ldots, N$, the following hold:
    \begin{enumerate}[{(i)}]
        \item $g_n$, $V_n$, and $F_n$ are convex functions,
        \item $n(\pi-K)^+\leq\max\{0,g_n(\pi)\}\leq V_n(\pi)\leq F_n(\pi)\leq n(1-K)\pi$.\label{lemmaFpropertiesii}
    \end{enumerate}
    \begin{proof}
We note that $g_1(\pi)=\pi-K$ is convex. By arguments for preservation of convexity for martingale diffusion processes in \cite{jansontysk-volatilitytime}, an expected value $\E_\pi[g_1(\Pi_t)]$ is convex in $\pi$ for every fixed time-point $t$ provided that $g_1$ is a convex function. Moreover, by a Bermudan approximation argument (see \cite{EKSTROM2004265}), preservation of convexity extends to the corresponding stopping problem, so $V_1(\pi)$ is convex. Then, by Jensen's inequality, $F_1\geq V_1$, and clearly $V_1\geq \pi-K$. Convexity of $F_1$ follows from convexity of $V_1$. 

Next, assume that $g_{n-1}, V_{n-1}$, and $F_{n-1}$ are convex. It follows that $g_n=\pi-K + F_{n-1}$ is also convex, and repeating the preservation of convexity and Bermudan approximation arguments yields that $V_n$ is convex, and so the convexity of $F_n$ follows. It follows that Jensen's inequality asserts $F_n\geq V_n$. That is, by induction, we have that $g_n$, $V_n$, and $F_n$ are convex for all $n$, and $g_n\leq V_n\leq F_n$. 
        
Moreover, since $(\pi-K)^+=\max\{g_1(\pi),0\}\leq (1-K)\pi$, we have 
\[
0\leq (\pi-K)^+ \leq V_1(\pi) \leq F_1(\pi) \leq (1-K)\pi.
\]
Assuming that 
\[0\leq (n-1)(\pi-K)^+\leq F_{n-1}(\pi) \leq (n-1)(1-K)\pi,\]
one sees that $n(\pi-K)^+\leq\max\{0,g_n(\pi)\}$ and $g_n=\pi-K +F_{n-1}(\pi)\leq n(1-K)\pi$, and then $V_n\leq F_n\leq n(1-K)\pi$.
The second statement thus follows by induction.
    \end{proof}
\end{lem}

\begin{rem}
Note that it follows from the bounds presented in Lemma~\ref{lemmaFproperties} (\ref{lemmaFpropertiesii}) that $V_n(0+)=F_n(0+)=0$ and $V_n(1-)=F_n(1-)=n(1-K)$. Moreover, Lemma~\ref{lemmaFproperties} implies that the first derivative of $V_n$ is bounded.
\end{rem}

For an illustration of the relationship between $V_1$ and $F_1$ presented in Lemma~\ref{lemmaFproperties}, see Figure~\ref{fig:lines}, where the function $F_1(\pi)$ produced numerically. For a short discussion on numerical methods used in this article, see Remark~\ref{rmk:numerics}.

\begin{figure}
    \centering
    \captionsetup{width=.8\textwidth}
    \makebox[\textwidth][c]{\includegraphics[width=0.7\textwidth]{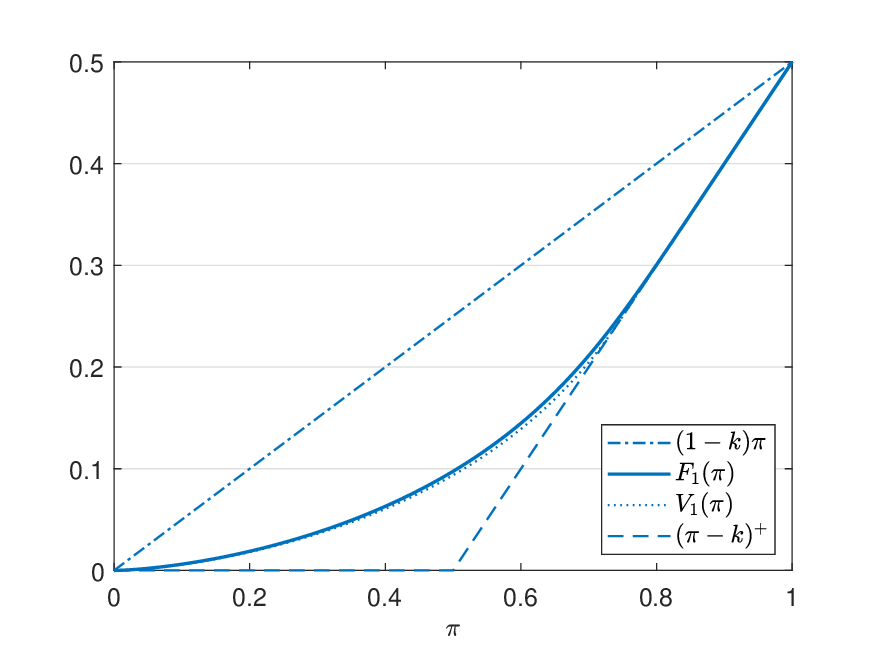}}
    \caption{Functions $V_1(\pi)$ and $F_1(\pi)$ , together with their shared upper bound $(1-K)\pi$ and lower bound $(\pi-K)^+$. Parameters used are: $\mu_0=-1$, $\mu_1=1$, $\sigma=4$, $r=0.1$, $N=10$, and $N\ep =1$.}
    \label{fig:lines}
\end{figure}

\subsection*{Case $n=1$}\label{sec:step1}
Consider
\begin{equation*}
V_1(\pi)=\sup_\tau \E\left[e^{-r \tau}\left(\Pi_\tau-K\right)\right]    
\end{equation*}
as given in equation~\eqref{valueftion}. Since this is a value function of a call option type (where $K$ is analogous to a strike price), we expect the optimal strategy to be given by a stopping time 
\[
\tau_1= \inf\{t\geq 0 : \Pi_t \geq b_1\}
\]
for some boundary $b_1$. By standard methods in optimal stopping theory and dynamic programming (see, for example, \cite{S1978}), one expects $V_1$ to solve the corresponding free-boundary problem 
\begin{equation}\label{step1-free-boundary-problem}
    \begin{cases}
        \left(\L V_1\right)(\pi)= 0, \quad & \text{ on } \pi \in (0,b_1), \\
        V_1(b_1) = b_1-K, & \\
        V_1'(b_1) = 1, &  \\
        V_1(0) = 0,
    \end{cases}
\end{equation}
where the differential operator $\L$ is given by 
\begin{equation}\label{step1generator}
    \L :=\frac{\rho^2\pi^2(1-\pi)^2}{2}\frac{\dif^2}{\dif\pi^2}-r.    
\end{equation}
The general solution to the ODE in \eqref{step1-free-boundary-problem} is of type 
\[
A_1 (1-\pi)\left(\frac{\pi}{1-\pi}\right)^{\gamma_+} + B_1 (1-\pi)\left(\frac{\pi}{1-\pi}\right)^{\gamma_-}
\]
for constants $A_1$ and $B_1$, where ${\gamma_+}$ and $\gamma_-$ are the positive and the negative solutions to the quadratic equation $\gamma^2-\gamma-\frac{2r}{\rho^2}=0$. From the boundary condition at $\pi=0$ we see that $B_1\equiv 0$. We denote 
\begin{equation}\label{G}
G(\pi) := (1-\pi)\left(\frac{\pi}{1-\pi}\right)^{\gamma_+},
\end{equation}
for which
\begin{equation}\label{Gbasicproperties}
    G(0)=0 \quad \text{ and } \left(\L G\right)(\pi)=0.
\end{equation}
Using this notation, the value function assumes the form 
\[
V_1(\pi) = A_1G(\pi)
\]
in the continuation region. Plugging in the boundary conditions at $b_1$ yields
\begin{equation}\label{step1_smootfitcond}
   \begin{cases}
    A_1G(b_1) = b_1-K \\
    A_1G'(b_1) = 1.
\end{cases} 
\end{equation}
By noting that
\begin{equation}\label{Gprime}
    G'(\pi) = \frac{{\gamma_+}-\pi}{\pi(1-\pi)}G(\pi),
\end{equation}
one uses the smooth fit equations~\eqref{step1_smootfitcond} to derive
\begin{equation}\label{c}
    b_1 = \frac{\gamma_+ K}{\gamma_+ + K - 1 }\geq K.
\end{equation}
This corresponds to the candidate value function assuming the form
\begin{equation}\label{step1candidate}
    \hat V_1(\pi) := 
    \begin{cases}
    \frac{b_1-K}{G(b_1)}G(\pi), \quad &\text{ if } \pi\in(0,b_1), \\
    \pi - K, \quad &\text{ if } \pi\in[b_1,1).
    \end{cases}
\end{equation}

Verifying that $\hat V_1 \equiv V_1$ is straightforward, see Proposition~\ref{step1verification} in Section~\ref{ch5}.

\subsection*{Case $1 < n \leq N$}\label{sec:stepn}

Now, for a general $n = 1,\ldots, N$, consider
\[
V_n(\pi) = \sup_\tau \E_\pi\left[e^{-r\tau}g_{n}(\Pi_\tau)\right] = \sup_\tau \E_\pi\left[e^{-r\tau}(\Pi_\tau-K+F_{n-1}(\Pi_\tau))\right]
\]
as given in~\eqref{valueftion}. Similarly as above, for each $n$, we expect the optimal stopping time to be of the form
\[
\tau_n = \inf\{t\geq 0: \Pi_t \geq b_n\}
\]
for some boundary $b_n$. In particular, we expect the value function $V_n$ to solve the free boundary problem
\begin{equation}\label{stepn-free-boundary-problem}
    \begin{cases}
        \left(\L V_n\right)(\pi) = 0, \quad & \text{ on } \pi \in (0,b_n), \\
        V_n(b_n) = g_{n}(b_n), & \\
        V_n'(b_n) = g'_{n}(b_n), & \\
        V_n(0) = 0.
    \end{cases}
\end{equation}
As in the case $n=1$ above, the smooth-fit guess gives us
\[
\begin{cases}
    A_nG(b_n) = g_{n}(b_n) \\
    A_nG'(b_n) = g_{n}'(b_n)
\end{cases}
\]
for a constant $A_n$ and $G(\pi)$ as in \eqref{G}. This yields an equation
\begin{equation}\label{stepn_smoothfiteqn}
G'(b_n)g_n(b_n) - G(b_n) g'_n(b_n) = 0.    
\end{equation}

If the smooth fit equation~\eqref{stepn_smoothfiteqn} admits a unique solution $b_n$, we then define a candidate value function as 
\begin{equation}\label{stepncandidate}
    \hat V_n(\pi) := 
    \begin{cases}
    \frac{g_n(b_n)}{G(b_n)}G(\pi), \quad &\text{ if } \pi\in(0,b_n), \\
    g_n(\pi), \quad &\text{ if } \pi\in[b_n,1).
    \end{cases}
\end{equation}

The relationship between $F_n$, $V_n$, $g_n$, and $F_{n-1}$ is illustrated in Figure~\ref{fig:lines2}. From the figure it can also be seen how the value functions $V_n$ coincide with respective payoff functions $g_n$ at the respective boundary points $b_n$ for $n=1,2,3$.

\begin{figure}[H]
    \centering
    \captionsetup{width=.8\textwidth}
    \makebox[\textwidth][c]{\includegraphics[width=0.7\textwidth]{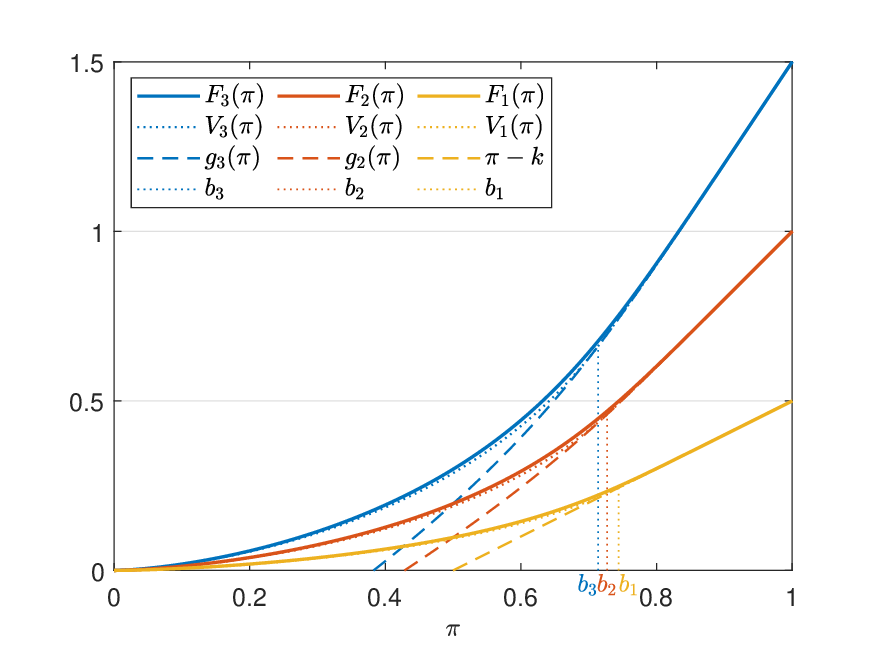}}
    \caption{Functions $V_n(\pi)$, $F_n(\pi)$, and $g_n(\pi)$ for $n=1,2,3$, and boundaries $b_3,b_2,$ and $b_1$. Parameters used are: $\mu_0=-1$, $\mu_1=1$, $\sigma=4$, $r=0.1$, $N=10$, and $N\ep =1$.}
    \label{fig:lines2}
\end{figure}
\section{Study of the smooth fit equations}\label{ch4}

We proceed by studying the solvability of equation~\eqref{stepn_smoothfiteqn}. In order to do so, we first need some technical results.

\begin{lem}\label{stepnlemmaF}
Let $f(\pi):=\E_\pi[v(Z(\ep,\pi))]$ for $\ep>0$ and some $C^2((0,1))$ function $v$ that has a bounded first derivative. 
\begin{equation}
   \text{If $\L v\leq 0$, then $\L f\leq 0$,}   
    \end{equation}
    and 
    \begin{equation}
     \text{if $\L v$ is decreasing in $\pi$, then $\L f$ is decreasing in } \pi.
    \end{equation}
    \end{lem}

    \begin{proof}
        For the first claim, let $\tilde f(t,\pi) = \E_\pi\left[e^{-r t} v(Z(t,\pi))\right]$ so that $f(\pi)=e^{r\ep}\tilde f(\ep,\pi)$. Since $\L v\leq 0$ and the first derivative of $v$ is bounded, $e^{-rt}v(Z(t,\pi))$ is a supermartingale,
        and $\tilde f(t,\pi)$ is decreasing in $t$. 
        Therefore, $\left(\L \tilde f\right) (t,\pi) = \tilde f_t(t,\pi) \leq 0$, and consequently
        \[
        (\L f)(\pi) = e^{r \varepsilon} \left(\L \tilde f \right)(\varepsilon, \pi)\leq 0.
        \]

        For the second claim, by Itô's formula we get 
        \begin{equation}\label{stepnlemmaF_ito}
            \tilde f(t,\pi) = \E_\pi\left[e^{-rt}v(Z(t,\pi))\right] = v(\pi) + \E_\pi\left[\int_0^t e^{-rs} (\L v)(\Pi_s)\dif s\right].
        \end{equation}
        Differentiating $\tilde f$ with respect to $t$ yields
        \[
            \tilde f_t(t,\pi) = e^{-rt} \E_\pi\left[(\L v)(\Pi_t)\right].
        \]
        Using a non-crossing property of paths (see \cite[Chapter IX.3]{RY}) and monotonicity of $\L v$, we find that 
        $\tilde f_{t\pi}(t,\pi)\leq 0$. Consequently
        \[
            \partial_\pi\left(\L f\right)(\pi) = e^{r\varepsilon} \partial_\pi\left(\L \tilde f\right)(\varepsilon,\pi) = e^{r\varepsilon}\partial_\pi \tilde f _t(\varepsilon,\pi) \leq 0,
        \]
        that is, $\left(\L f\right)(\pi)$ is decreasing in $\pi$. 
    \end{proof}

\begin{rem}\label{rmk:c2valueftion}
    In the following, we apply Lemma~\ref{stepnlemmaF} to the function $v=V_n$ which is not in $C^2$ but merely in $C^1((0,1)) \cap C^2((0,b_n)\cup (b_n,1))$. However, a closer inspection of the proof of the lemma shows that the conclusion also holds in this case.  
\end{rem}

To simplify the presentation of our technical results, for the remainder of Section~\ref{ch4}, we work under the following assumption. Note that Assumption~\ref{ass} is not required to derive the main results presented in Section~\ref{ch5}.
\begin{assumpt}\label{ass}
For a given $n\geq2$, we assume that $\L V_{n-1}(\pi)\le 0$ and $\L V_{n-1}(\pi)$ is decreasing.
\end{assumpt}
That is, we proceed to show that under Assumption~\ref{ass}, for each $n\in\{2,\ldots,N\}$, the stopping problem with a one-sided stopping boundary determined by smooth fit is solved. 

Lemma~\ref{stepnlemmaF} implies the following result, for which recall that
    \[
    (\L F_{n})(\pi) = \frac{\rho^2 \pi^2 (1-\pi)^2}{2}F''_n(\pi) - rF_n(\pi).
    \]
\begin{prop}\label{stepn-prop-g}
Assume that Assumption~\ref{ass} holds. Then $\left(\L g_n\right)(\pi)$ is strictly decreasing. Moreover, there exist unique solutions $\pi_n^0$ and $\pi_n^*$ of $g_n(\pi^0_n) = 0$ and $\left(\L g_n\right)(\pi_n^*)=0$, respectively. Furthermore, $\pi_n^*\in(0,K]$ and $\pi_n^0\leq \pi_n^*\leq K$.

    \begin{proof}
        By Lemma~\ref{stepnlemmaF}, if $\left(\L V_{n-1}\right)(\pi)$ is decreasing then $(\L F_{n-1})(\pi)$ is decreasing in $\pi$, and thus
        \begin{equation}\label{stepnlemmaF_neq}
        (\L g_n)(\pi) = -r(\pi-K)+(\L F_{n-1})(\pi)
        \end{equation}
        is strictly decreasing. In addition, if $g_n(\pi)< 0$, we have
        \begin{align*}
        (\L g_n)(\pi) &= -r(\pi-K) + \frac{\rho^2 \pi^2 (1-\pi)^2}{2}F''_{n-1}(\pi) - rF_{n-1}(\pi) \\&= \frac{\rho^2 \pi^2 (1-\pi)^2}{2}F''_{n-1}(\pi) - rg_n(\pi)> 0.
        \end{align*}
        Moreover, by~\eqref{stepnlemmaF_neq} we have that, at $\pi=K$,
        \[
        (\L g_n)(K) = (\L F_{n-1})(K) \leq 0.
        \]
        That is, $(\L g_n)(\pi)$ is strictly decreasing, and it  satisfies $(\L g_n)(\pi)> 0$ for small $\pi$. Moreover, at $K$ it is non-positive, which shows that a unique solution $\pi_n^*$ to $\left(\L g_n\right)(\pi^*_n)=0$ exists, and also that $\pi_n^*\in(0,K]$. 
        
        To show the remaining claim, it suffices to note that 
        \[
        (\L g_n)(\pi_n^0) = \frac{\rho^2 {(\pi_n^0)}^2 \left(1-{\pi_n^0}\right)^2}{2}F''_{n-1}(\pi_n^0)\geq 0,
        \]
        so $\pi_n^0\leq \pi_n^*$.
    \end{proof}
\end{prop}

To show that the equation~\eqref{stepn_smoothfiteqn} indeed has a unique solution, we define 
\begin{equation}\label{stepn_h}
    h_n(\pi) := G'(\pi)g_n(\pi) - G(\pi)g'_n(\pi)
\end{equation}
for $\pi\in(0,1)$. Recall that we expect that the boundary solving the $n$th free-boundary problem~\eqref{stepn-free-boundary-problem} is a solution to the equation $h_n(\pi)=0$.

\begin{prop}\label{stepn_prop_uniquebn}
   Assume that Assumption~\ref{ass} holds. Then, there exists a unique solution $b_n$ to the equation $h_n(b_n)=0$. Moreover, $b_n\in(\pi^*_n,b_1]$.
    \end{prop}

    \begin{proof}
        Since $G(\pi)g'_n(\pi)\geq 0$ for all $\pi\in(0,b_1)$, we have $h_n(\pi)< 0$ for $\pi\in(0,\pi^0_n)$, where $\pi^0_n$ is the solution to $g_n(\pi^0_n) = 0$. 

        Using $(\L G)(\pi) = 0$, we have
        \begin{align*}
        \frac{\rho^2\pi^2(1-\pi)^2}{2}h_n'(\pi) &= g_n(\pi)\frac{\rho^2\pi^2(1-\pi)^2}{2}G''(\pi)-G(\pi)\frac{\rho^2\pi^2(1-\pi)^2}{2}g_n''(\pi) \\
        &= -G(\pi)(\L g_n)(\pi)\quad
        \begin{cases}
            \leq 0 \quad \text{ if } \pi<\pi_n^*\\
            \geq 0 \quad \text{ if } \pi>\pi_n^*,
        \end{cases} 
        \end{align*}
        where the last inequality pair comes directly from Proposition~\ref{stepn-prop-g}. The sign of
        \[
        \frac{\rho^2\pi^2(1-\pi)^2}{2}h_n'(\pi)
        \]
        coincides with the sign of $h_n'(\pi)$. Consequently, $h_n(\pi)$ is decreasing on $(0,\pi_n^*)$ and increasing on $(\pi_n^*,1)$, so there exists at most one solution of $h_n(b_n)=0$, and for such a solution we must have $b_n>\pi^*_n$.
        We next show that $h_n(b_1)\geq 0$, which then finishes the proof.
        
        To see that $h_n(b_1)\geq 0$, select a constant $D$ such that $D G(b_1) - F_{n-1}(b_1)=0$. Lemma~\ref{stepnlemmaF},
        together with properties of $F_{n-1}$ and $G$ (see Lemma~\ref{lemmaFproperties} and equation~\eqref{Gbasicproperties}, respectively), yields
        \[
        \begin{cases}
            D G(0)-F_{n-1}(0)=0,\\
           D G(b_1)-F_{n-1}(b_1)=0,\\
            \L\left(DG-F_{n-1}\right)(\pi) \geq 0 \quad \text{ for } \pi\in(0,b_1).
        \end{cases}
        \]
        By the maximum principle, it follows that $D G(\pi)-F_{n-1}(\pi)\leq0$ for $\pi\in(0,b_1)$, so 
        \[
        DG'(b_1)-F'_{n-1}(b_1)\geq 0.    
        \]
        Consequently, $G'(b_1)F_{n-1}(b_1)-G(b_1)F_{n-1}'(b_1)\geq 0$, so   
        \begin{eqnarray*}
        h_n(b_1) &=& G'(b_1)(b_1-K+F_{n-1}(b_1))-G(b_1)(1+F_{n-1}'(b_1))\\
        &=& G'(b_1)F_{n-1}(b_1)-G(b_1)F_{n-1}'(b_1)
        \geq  0,
        \end{eqnarray*}
where the last equality comes from noting that $G'(\pi)(\pi-K)-G(\pi)=0$ is solved by $\pi=b_1$ by~\eqref{c}. This completes the proof.
    \end{proof}

\section{Main results}\label{ch5}

We start by verifying the candidate value function $\hat V_1(\pi)$.

\begin{prop}\label{step1verification}
    Let $V_1(\pi)$ be as in~\eqref{valueftion} and $\hat V_1(\pi)$ as in~\eqref{step1candidate}. Then $V_1(\pi)\equiv \hat V_1(\pi)$ for all $\pi\in(0,1)$. Moreover, the stopping time $\tau_1 = \inf \{t\geq 0 : \Pi_t\geq b_1\}$, where $b_1$ is given by~\eqref{c}, is optimal. 
    \begin{proof}
        When $\pi\geq b_1$, we directly have $\hat V_1(\pi)\geq \pi-K$ by \eqref{step1candidate}.
        By the convexity of $\hat V_1(\pi)$, it follows that $\hat V_1(\pi)\geq \pi-K$ also for $\pi< b_1$ and, in particular,
        \begin{equation}\label{step1verif1}
            \hat V_1(\pi)\geq \pi-K \quad \text{ for all } \pi\in(0,1).
        \end{equation}
        Since $b_1\geq K$, when $\pi> b_1$ we have
        \begin{equation}\label{step1verif2}
            \left(\L \hat V_1\right)(\pi) = \frac{\rho^2 \pi^2(1-\pi)^2}{2}  \hat V_1''(\pi) - r \hat V_1(\pi) \leq -r(b_1-K) \leq 0.            
        \end{equation}
        Similarly, if $\pi\leq b_1$, by \eqref{step1-free-boundary-problem} the same inequality holds, and so
        \begin{equation}\label{step1verif3}
            \left(\L \hat V_1\right)(\pi) \leq 0  \quad \text{ for all } \pi \in (0,1).
        \end{equation}
        Therefore, by a standard verification argument (see, for example, \cite{PS} for theory and several examples), we indeed have $V_1 (\pi)\equiv \hat V_1(\pi)$. The threshold type of stopping time $\tau_1$ follows the threshold type of $\hat V_1(\pi)$ and is ensured by~\eqref{step1verif1} (see, for example,~\cite{villeneuvestopping}). Together with~\eqref{step1verif3}, a standard verification argument then shows that $V_1(\pi)$ is attained at $\tau_1$.
    \end{proof}
\end{prop}

Next, we provide a verification result for $\hat V_n(\pi)$ for an $n\in\{2,\ldots,N\}$. 

\begin{prop}\label{stepnverification}
    Assume that Assumption~\ref{ass} holds for some $n\in\{2,\ldots,N\}$.
    Let $V_n(\pi)$ be as in~\eqref{valueftion} and $\hat V_n(\pi)$ as in~\eqref{stepncandidate}. Then $V_n(\pi)\equiv \hat V_n (\pi)$ for all $\pi\in(0,1)$. Moreover, let $\tau_n = \inf \{t\geq 0 : \Pi_t\geq b_n\},$ where $b_n$ is the unique solution of $h_n(b_n)=0$ (see equation~\eqref{stepn_h}). Then, the stopping time $\tau_n$ is optimal.
    \begin{proof}
        Choose a fixed $n\in\{2,\ldots,N\}$. Similarly as in the proof of Proposition~\ref{step1verification}, for the verification argument we need
        \begin{equation}\label{stepnverif1}
            \left(\L \hat V_n\right)(\pi) \leq 0 \quad \text{ for all }  \pi \in (0,1)
        \end{equation}
        and 
        \begin{equation}\label{stepnverif2}
            \hat V_n(\pi) \geq g_n(\pi) \quad \text{ for all } \pi \in (0,1).
        \end{equation}
        For the condition~\eqref{stepnverif1}, we note that for $\pi< b_n$ we automatically have
        \[
            \left(\L \hat V_n\right)(\pi) = \frac{g_n(b_n)}{G(b_n)}(\L G)(\pi) = 0.
        \]
        On the other hand, when $\pi\geq b_n$, we have
        \[
        \left(\L \hat V_n \right)(\pi) = \left(\L g_{n}\right)(\pi) \leq 0
        \]
        since $b_n>\pi_n^*$ (see Propositions~\ref{stepn-prop-g} and~\ref{stepn_prop_uniquebn}). 
        
        For the condition~\eqref{stepnverif2} we note that if $\pi>b_n$, then by construction we have $\hat V_n (\pi) = g_n(\pi)$. For $\pi\leq b_n$ we argue as follows. 

        First, we claim that $\hat V_n>g_n$ on $[\pi^*_n,b_n)$. In fact, if this was not the case, then there exists 
        $a\in[\pi^*_n,b_n)$ with $\hat V_n(a)\leq g_n(a)$. 
        Since $\L g_n\leq \L\hat V_n=0$ on $[a,b_n]$, the maximum principle yields $V_n\leq g_n$ on $[a,b_n]$. On the other hand, $(\L g_n)(b_n)<0=
        (\L\hat V_n)(b_n)$, which implies that $\hat V_n>g_n$ in a left neighborhood of $b_n$, which is a contradiction. It follows that 
        $\hat V_n>g_n$ on $[\pi^*_n,b_n)$.
        
        Second, we show that $\hat V_n\geq g_n$ on $(0,\pi^*_n)$.
        
        Since $\hat V_n(0)\geq g_n(0)$, $\hat V_n(\pi^*_n)> g_n(\pi^*_n)$ and $\mathcal Lg_n\geq \mathcal L\hat V_n=0$ on $(0,\pi^*_n)$, the maximum principle gives $\hat V_n\geq g_n$ also on $(0,\pi^*_n)$.
        
        Since conditions~\eqref{stepnverif1} and~\eqref{stepnverif2} hold, by standard verification arguments (see note in the proof of Proposition~\ref{step1verification}) we have $\hat V_n (\pi)\equiv V_n(\pi)$ and $V_n(\pi)$ is attained at $\tau_n$.
    \end{proof}
\end{prop} 

To combine the individual verification results~\ref{step1verification} and~\ref{stepnverification} to our main result, we need the following simple proposition. 

\begin{prop}\label{prop-inductionstep}
Assume that Assumption~\ref{ass} holds for some $n\in\{2,\ldots,N\}$. Then, also $(\L V_n)(\pi)$ is decreasing.
\begin{proof}
By the construction of the candidate value function in~\eqref{stepncandidate} and Proposition~\ref{stepnverification}, we have $(\L V_n)(\pi)=0$ for $\pi<b_n$, and $(\L V_n)(\pi)=(\L g_n)(\pi)$ for $\pi>b_n$. Therefore, by Proposition~\ref{stepn-prop-g}, $(\L V_n)(\pi)$ is decreasing.
\end{proof}
\end{prop}

The following theorem contains the main result of our article. We use an induction argument that, following Proposition~\ref{prop-inductionstep}, validates Assumption~\ref{ass}, making it not necessary for the theorem.

\begin{thm}\label{mainthm}
    Let $V_1(\pi)$ and $V_n(\pi)$ be as given in equation~\eqref{valueftion}, and $\hat V_1(\pi)$ and $\hat V_n(\pi)$ be as given in equations~\eqref{step1candidate} and~\eqref{stepncandidate}, respectively. Then, $(\L V_{n})(\pi)$ is monotone decreasing and
    \begin{equation}\label{mainthmequation}
        \quad \hat V_n(\pi)\equiv V_n(\pi)
    \end{equation}
    for all $\pi\in(0,1)$ and for all $n\in\{1,\ldots,N\}$. Moreover, for each $n$, the optimal investment strategy is to invest at random times
    \begin{equation*}
    \begin{cases}
        \tau_1 = \inf \{t\geq 0 : \Pi_t\geq b_1\} \\
        \tau_n = \inf \{t\geq 0 : \Pi_t\geq b_n\},
    \end{cases}
    \end{equation*}
    characterized by a sequence of stopping boundaries $\{b_n\}_{n=1}^{N}$, where $b_1$ is given by the equation~\eqref{c}, and for each $n\in\{2,\ldots,N\}$, the boundary $b_n$ is the unique solution of $h_n(b_n)=0$, where $h_n$ is given in equation~\eqref{stepn_h}.
    \begin{proof}
    Claim $\hat V_1(\pi)\equiv V_1(\pi)$ comes directly from Proposition~\ref{step1verification}. It follows that 
    \[
    (\L V_1)(\pi) =
    \begin{cases}
        -r(\pi-K), \quad &\pi \geq b_1\\
        0, \quad &\pi <b_1  
    \end{cases}
    \]
    is decreasing and so by Proposition~\ref{prop-inductionstep}, $(\L V_2)(\pi)$ is also decreasing. Then, for any given $n>2$, we assume that $(\L V_{n-1})(\pi)$ is decreasing. This assumption, together with Proposition~\ref{prop-inductionstep}, implies that also $(\L V_{n})(\pi)$ is decreasing. That is, by induction, $(\L V_{n})(\pi)$ is decreasing for all $n$ and so Proposition~\ref{stepnverification} asserts that $\hat V_n(\pi)\equiv V_n(\pi)$ for all $n\in\{2,\ldots,N\}$. 
    
    Optimality of the stopping times $\{\tau_n\}_{n=1}^{N}$ follows the verification of the threshold type value function with $\hat V_n(\pi)\equiv V_n(\pi)$, and it is directly inherited from Propositions~\ref{step1verification} and~\ref{stepnverification}.

    \end{proof}
\end{thm}

An example of the optimal strategy characterized as boundaries $\{b_n\}_{n=1}^{N}$ is illustrated in the following figure.  

\begin{figure}[H]
    \centering
    \captionsetup{width=.8\textwidth}
    \makebox[\textwidth][c]{\includegraphics[width=0.7\textwidth]{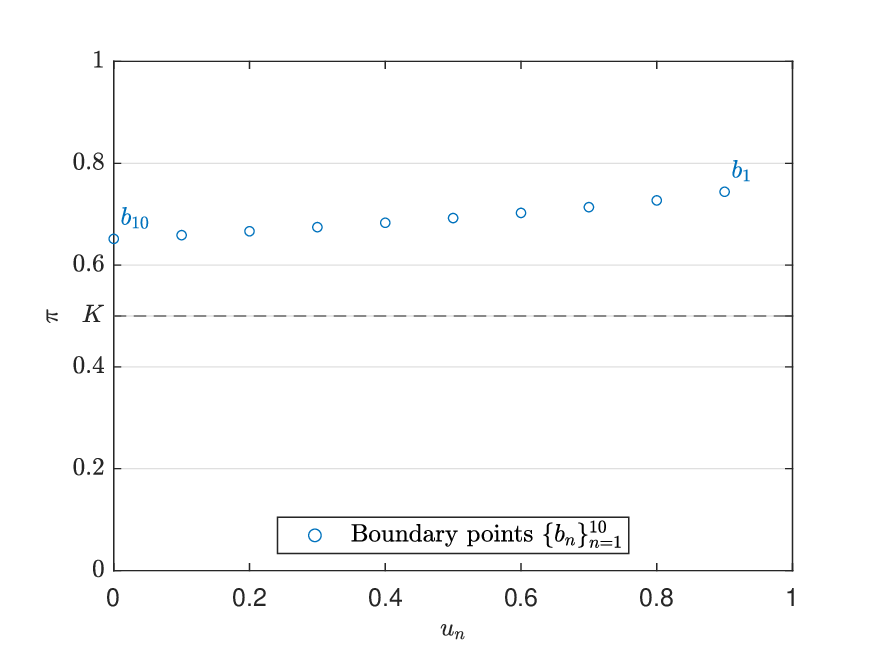}}
    \caption{Points $(u_1,b_{10}),\ldots,(u_{10},b_1)$ that illustrate the set of stopping boundaries $b_{10},\ldots,b_1$. The boundaries $\{b_n\}_{n=2}^{10}$ are solved numerically whereas the boundary $b_1$ is given in \eqref{c}. Parameters used are: $\mu_0=-1$, $\mu_1=1$, $\sigma=4$, $r=0.1$, $N=10$, and $N\ep =1$. From the choice of $\mu_0$, $\mu_1$, and $\sigma$, the values $K=0.5$ and $\rho=0.5$ follow. The particular boundaries presented in this figure are used as benchmark boundaries in the figures that appear Section~\ref{ch6}.}
    \label{fig:boundaries}
\end{figure}

We finish the study of the \textit{learning-from-the-past} effect by discussing some properties of boundaries $\{b_n\}_{n=1}^N$ with numerical methods.

\section{Sensitivity analysis}\label{ch6}

We conduct a numerical study on the behavior of a sequence of boundaries $\{b_n\}_{n=1}^N$ with respect to changes in key model parameters. To highlight some of the results, we show that a case with $b_n<K$ is possible (see Figure~\ref{fig:epcomparison}), which shows that it may be optimal to invest in a project with negative expected value (similar observations have been made in \cite{LBD}). In addition, we show that $\{b_n\}_{n=1}^{N}$ is not necessarily elementwise monotone with respect to $\mu_1$ (see Figure~\ref{fig:mucomparison}), demonstrating a mixed effect between $K$ and signal-to-noise ratio $\rho$. We then conclude with an observation that increasing $N$ (and so decreasing learning per individual investment) decreases boundaries $\{b_n\}_{n=1}^{N}$ elementwise (see Figure~\ref{fig:largeN}).

\begin{rem}\label{rmk:numerics}
Recall that boundaries $\{b_n\}_{n=1}^{N}$ are solved from \eqref{stepn_smoothfiteqn}. To numerically solve for these boundaries, functions $F_n$ are solved using a finite differences method. In particular, we consider a second-order partial differential equation
\[\begin{cases}
        (\tilde F_n)_t(t,\pi) = \frac{\rho^2\pi^2(1-\pi)^2}{2}(\tilde F_n)_{\pi\pi}(t,\pi)\\
        \tilde F_n(0,\pi) = V_n(\pi)
    \end{cases}\]
with $\tilde F_n(t,\pi):=E_\pi\left[V_n(Z(t,\pi))\right]$. In particular, $ \tilde F_n(\ep,\pi)\equiv F_n(\pi)$. Moreover, we establish boundary values of $\tilde F_n(t,\pi)$ at $\pi=0$ and $\pi=1$ using the results in Lemma~\ref{lemmaFproperties}. The boundaries $b_n(\pi)$ are solved on a discrete $\pi$ grid but the values plotted and used in recursion are taken as weighted averages between the grid points. In all figures, the numerically solved boundaries $\{b_n\}_{n=1}^{N}$ are produced using the same set of parameters unless otherwise mentioned. Boundaries $b_1$ are also solved numerically and they are verified with the explicit values given in~\eqref{c}.  
\end{rem}

\begin{figure}[H]
    \centering
    \captionsetup{width=.8\textwidth}
    \makebox[\textwidth][c]{\includegraphics[width=0.7\textwidth]{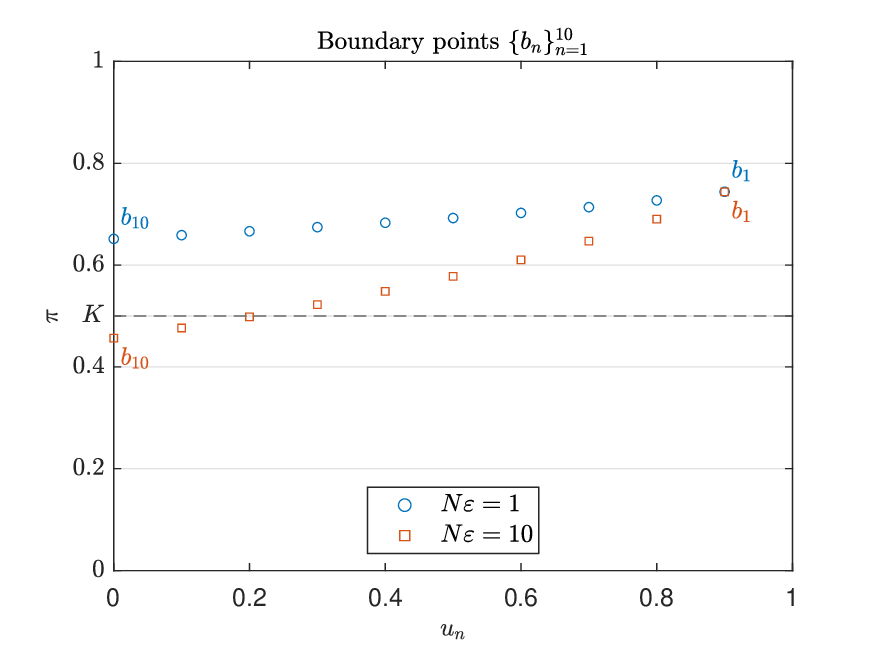}}
    \caption{Comparison of two sets of boundaries $b_{10},\ldots,b_1$ with different total learning rates $N\ep$. Parameters used are: $\mu_0=-1$, $\mu_1=1$, $\sigma=4$, $r=0.1$, $N=10$, and $N\ep =1$ and $N\ep=10$, respectively. The boundaries denoted by (blue) circles correspond to the exact boundaries presented in Figure~\ref{fig:boundaries}.}
    \label{fig:epcomparison}
\end{figure}

In Figure~\ref{fig:epcomparison}, we plot the boundaries for two different levels of $N\ep$, which denotes the amount of total additional learning available in $N$ investments (see Remark~\ref{rmk:learning} for discussion). Note how the boundaries $b_1$ coincide for both $N\ep=1$ and $N\ep=10$. In the case of $N\ep=10$, we have $b_{10}<b_9<b_8<K$. This suggests a tradeoff between learning and earning for early investments, a prevalent topic in many studies of incomplete information (see, for example,~\cite{HKZ} for a classical study in a Bayesian setting).

Next, we compare the boundaries for different levels of $\sigma$ (the diffusion coefficient of the observation process~\eqref{observationprocess}), which affects dynamics of the process $\Pi_t$ (described in~\eqref{pidynamics}) via the signal-to-noise ratio $\rho = \frac{\mu_1-\mu_0}{\sigma}$. It is expected that the boundaries $\{b_n\}_{n=1}^{10}$ are elementwise increasing with respect to $\rho$ and so decreasing with respect to $\sigma$: a higher signal-to-noise ratio implies that the investor learns more only by observing $X_t$, whereas a lower signal-to-noise ratio makes the investor more eager to invest to attain additional learning. This phenomenon is confirmed in Figure~\ref{fig:sigmacomparison}.

\begin{figure}
    \centering
    \captionsetup{width=.8\textwidth}
    \makebox[\textwidth][c]{\includegraphics[width=0.7\textwidth]{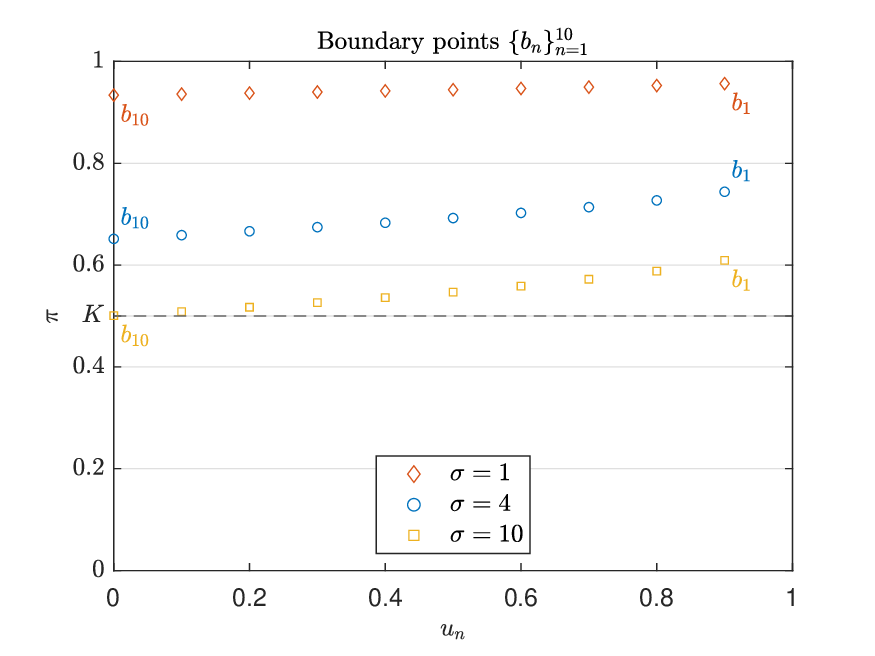}}
    \caption{Comparison of three sets of boundaries $b_{10},\ldots,b_1$ with different values of $\sigma$. Parameters used are: $\mu_0=-1$, $\mu_1=1$, $r=0.1$, $N=10$, $N\ep =1$, and $\sigma=1$, $\sigma=4$, and $\sigma=10$, respectively corresponding to signal-to-noise ratios of $\rho=2$, $\rho=0.5$, and $\rho=0.2$. The boundaries denoted by (blue) circles correspond to the exact boundaries presented in Figure~\ref{fig:boundaries}.}
    \label{fig:sigmacomparison}
\end{figure}

\begin{figure}
    \centering
    \captionsetup{width=.8\textwidth}
    \makebox[\textwidth][c]{\includegraphics[width=0.7\textwidth]{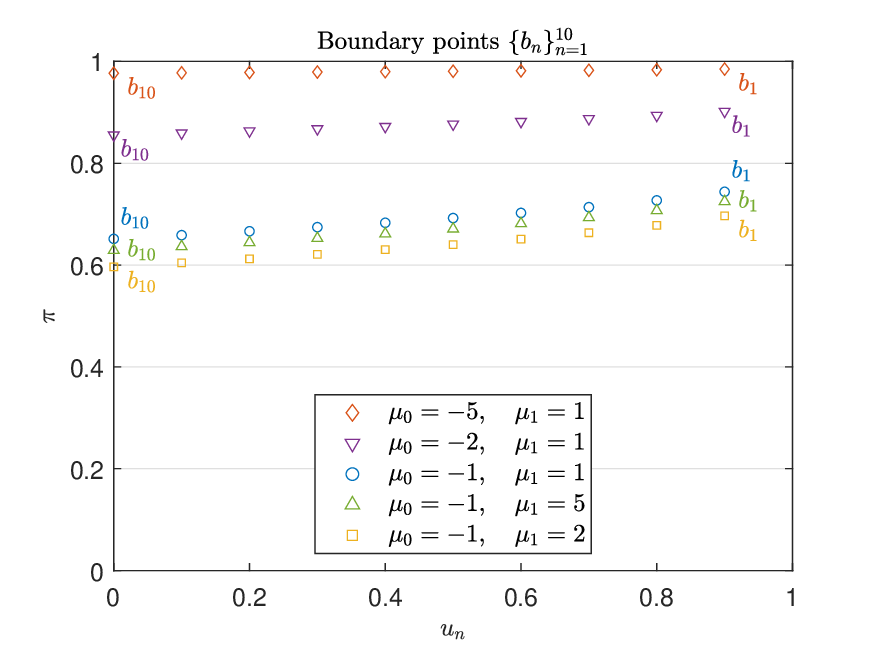}}
    \caption{Comparison of five sets of boundaries $b_{10},\ldots,b_1$ with different values for the parameters $\mu_0$ and $\mu_1$, choice of which directly affect $K$ and $\rho$. Parameters used were: $\sigma=4$, $r=0.1$, $N=10$, $N\ep =1$, and the pairs $(-5,1)$, $(-2,1)$, $(-1,1)$, $(-1,2)$, and $(-1,5)$ of the parameter values $(\mu_0,\mu_1)$, which respectively imply values $(\frac{5}{6},\frac{3}{2})$, $(\frac{2}{3},\frac{3}{4})$, $(\frac{1}{2},\frac{1}{2})$, $(\frac{1}{3},\frac{3}{4})$, and $(\frac{1}{6},\frac{3}{2})$ for pairs $(K,\rho)$. The boundaries denoted by (blue) circles correspond to the exact boundaries presented in Figure~\ref{fig:boundaries}.}
    \label{fig:mucomparison}
\end{figure}

It is also noteworthy to compare boundaries for different project values $\mu_0$ and $\mu_1$. Recall that in Section~\ref{ch2} we motivated the solution concept to our problem by defining
\[
K=\frac{-\mu_0}{\mu_1-\mu_0} \quad \text{ and } \rho=\frac{\mu_1-\mu_0}{\sigma}.
\]
That is, comparing stopping boundaries under different project values $\mu_0$ and $\mu_1$ reduces to comparing boundaries under different values of $K$ and $\rho$. We expect that increasing $K$ pushes up the boundaries, and in Figure~\ref{fig:sigmacomparison} it is demonstrated that the boundaries $\{b_n\}_{n=1}^{10}$ also are elementwise monotone with respect to $\rho$. Moreover, both $K$ and $\rho$ are monotone with respect to $\mu_0$. However, $K$ and $\rho$ have different monotonities in $\mu_1$, leading to a mixed effect. This is demonstrated in Figure~\ref{fig:mucomparison}, where there is no elementwise monotonicity of boundaries $\{b_n\}_{n=1}^{10}$ with respect to chances in $\mu_1$.

Similar comparison can be done for the discount rate $r$. Intuitively, high discount rate penalizes waiting which is tied to earlier investment times. This is confirmed in Figure~\ref{fig:rcomparison}, where we compare boundaries for different levels of $r$.

\begin{figure}
    \centering
    \captionsetup{width=.8\textwidth}
    \makebox[\textwidth][c]{\includegraphics[width=0.7\textwidth]{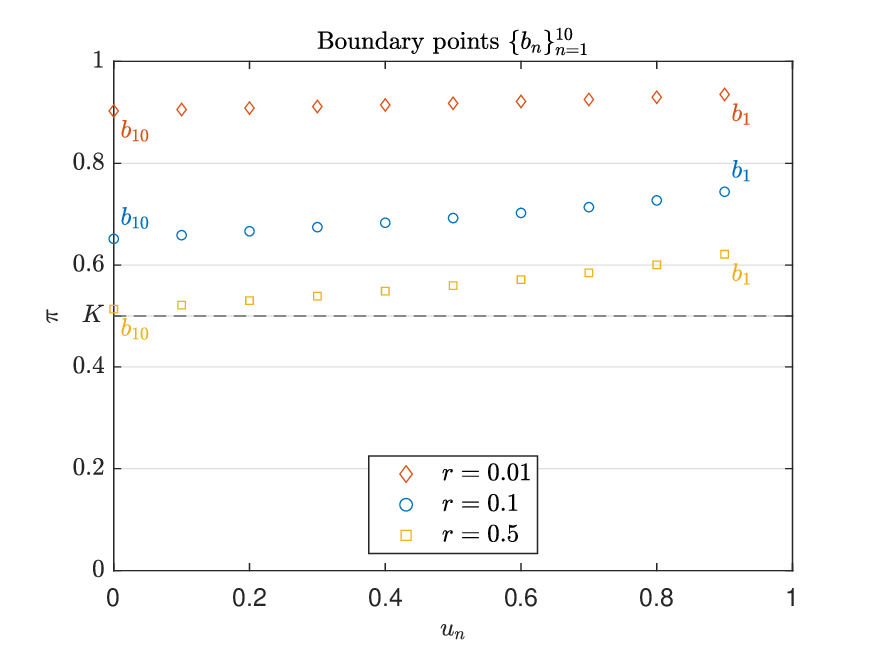}}
    \caption{Comparison of three sets of boundaries $b_{10},\ldots,b_1$ with different values for discount rate $r$. Parameters used were: $\mu_0=-1, \mu_1=1, \sigma=4, N=10$, $N\ep =1$, and $r=0.01$, $r=0.1$ and $r=0.5$, respectively. The boundaries denoted by (blue) circles correspond to the exact boundaries presented in Figure~\ref{fig:boundaries}.}
    \label{fig:rcomparison}
\end{figure}

Increasing $N$ should also affect the boundaries. Such comparison is possible with fixing $N\ep$, i.e. to scale inversely the amount of additional learning $\ep$ with $N$. One expects increasing $N$ to elementwise decrease $\{b_n\}_{n=1}^{N}$, reducing the relative learning effect of an individual investment. This is confirmed in Figure~\ref{fig:largeN}.

\begin{figure}
    \centering
    \captionsetup{width=.80\textwidth}
    \makebox[\textwidth][c]{\includegraphics[width=0.7\textwidth]{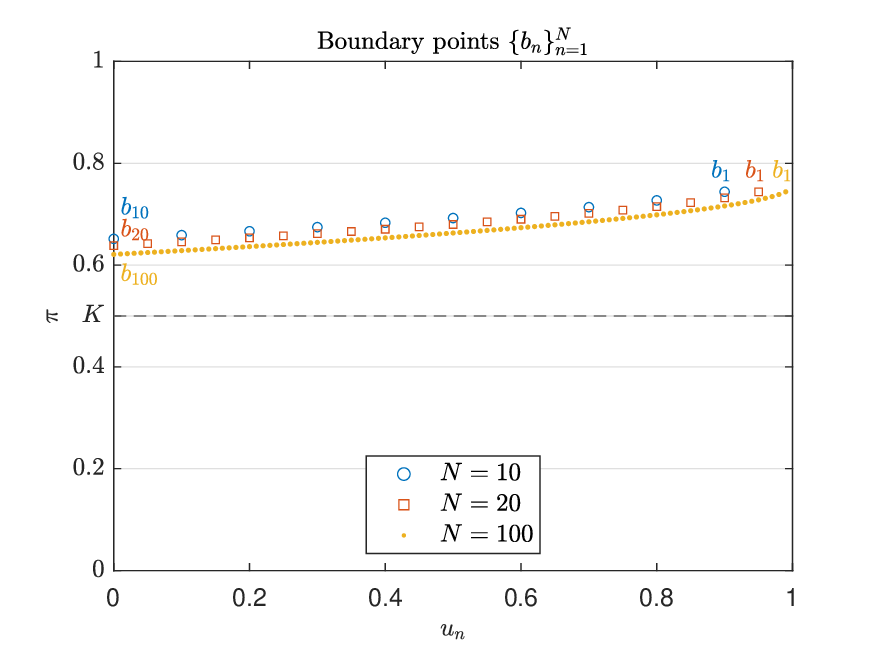}}
    \caption{Three sequences of boundaries $b_N,\ldots,b_1$ with different total number of exercise rights $N$. Parameters used are: $\mu_0=-1$, $\mu_1=1$, $\sigma=4$, $r=0.1$, $N\ep =1$, and $N=10$, $N=20$ and $N=100$, respectively. The boundaries denoted by (blue) circles correspond to the exact boundaries presented in Figure~\ref{fig:boundaries}.}
    \label{fig:largeN}
\end{figure}

\pagebreak

We finish with the following remark on possible further extensions to our research.

\begin{rem}
We note that solving for the multiple stopping problem~\eqref{valueftion} indeed is a discrete analogue of solving the continuous problem~\eqref{introdiffusioneq}--\eqref{introvalueftion}. By taking the limit $N\to\infty$ and by detaching from the discrete grid in investment levels and optimal stopping times, we expect the problem to collapse into a continuous stochastic control problem, where one attains a continuous boundary $b$ corresponding to optimal control of the process $(U_t)_{t\geq 0}$, formal definition of which we leave for further research. Indeed, definition of the admissible class of controls and characterizing the boundary $b$ through suitable boundary conditions remains to be an ample opportunity for future research.
\end{rem}

\begin{acknowledgement}
We sincerely thank Erik Ekström for his patient, kind, and useful guidance and the countless discussions we had which were invaluable in shaping this article to its current form.
\end{acknowledgement}

\end{document}